\documentclass{article}[11pt]

\usepackage[ruled,vlined]{algorithm2e}

\usepackage{graphicx}
\usepackage{amsmath,amsfonts,amssymb,latexsym}
\usepackage{enumerate}
\usepackage{setspace}
\usepackage{color}
\usepackage[cp1250]{inputenc}
\usepackage{tikz}
\usepackage{soul}

\newtheorem{thm}{Theorem}
\newtheorem{prop}[thm]{Proposition}
\newtheorem{ob}[thm]{Observation}

\newtheorem{cor}[thm]{Corollary}

\newtheorem{lemma}[thm]{Lemma}

\newtheorem{prob}{Problem}

\newcommand{\QED}{$\Box$}
\newcommand{\Wd}{{\rm Wd}}

\newcommand{\1}{ \vspace{0.1cm} }

\newcommand{\pn}{{\rm pn}}
\newcommand{\epn}{{\rm epn}}
\newcommand{\ipn}{{\rm ipn}}

\def\vertex(#1){\put(#1){\circle*{1.8}}}
\def\lab(#1)#2{\put(#1){\makebox(0,0)[c]{#2}}}

\newcommand{\grt}{\gamma_{\rm gr}^t}
\newcommand{\grz}{\gamma_{\rm gr}^{\rm Z}}

\let\oldenumerate\enumerate
\renewcommand{\enumerate}{
  \oldenumerate
  \setlength{\itemsep}{0pt}
  \setlength{\parskip}{0pt}
  \setlength{\parsep}{0pt}
}

\begin{document}

\title{Bounds on zero forcing using (upper) total domination and minimum degree}

\author{
Bo\v{s}tjan Bre\v{s}ar$^{a,b}$
\and
Mar\'{i}a Gracia Cornet$^{c,d}$
\and Tanja Dravec$^{a,b}$
\and Michael Henning$^{e}$
}

\date{\today}

\maketitle

\begin{center}
$^a$ Faculty of Natural Sciences and Mathematics, University of Maribor, Slovenia \\
$^b$ Institute of Mathematics, Physics and Mechanics, Ljubljana, Slovenia\\
$^c$  Depto. de Matem\'atica, FCEIA, Universidad Nacional de Rosario, Argentina \\
$^d$ Consejo Nacional de Investigaciones Cient\'ificas y T\'ecnicas, Argentina \\
$^e$ Department of Mathematics and Applied Mathematics, University of Johannesburg,  South Africa\\
\end{center}

\begin{abstract}
While a number of bounds are known on the zero forcing number $Z(G)$ of a graph $G$ expressed in terms of the order of a graph and maximum or minimum degree, we present two bounds that are related to the (upper) total domination number $\gamma_t(G)$ (resp.~$\Gamma_t(G)$) of $G$. We prove that $Z(G)+\gamma_t(G)\le n(G)$ and $Z(G)+\frac{\Gamma_t(G)}{2}\le n(G)$ holds for any graph $G$ with no isolated vertices of order $n(G)$. Both bounds are sharp as demonstrated by several infinite families of graphs. In particular, we show that every graph $H$ is an induced subgraph of a graph $G$ with $Z(G)+\frac{\Gamma_t(G)}{2}=n(G)$. Furthermore, we prove a characterization of graphs with power domination equal to $1$, from which we derive a characterization of the extremal graphs attaining the trivial lower bound $Z(G)\ge \delta(G)$. The class of graphs that appears in the corresponding characterizations is obtained by extending an idea from [D.D.~Row, A technique for computing the zero forcing number of a graph with a cut-vertex, Linear Alg.\ Appl.\ 436 (2012) 4423--4432], where the graphs  with zero forcing number equal to $2$ were characterized.
\end{abstract}


\newpage
\section{Introduction}
\label{sec:intro}

Concentrating solely on the positions of non-zero entries of real symmetric matrices, they can be described by the adjacency matrix of an undirected graph. Fixing the positions of non-zero values and considering all possible values, the concepts of maximum nullity and minimum rank of the corresponding graph arise.  The zero forcing number was introduced in~\cite{AIM} as a useful (and often attained) upper bound on the maximum nullity of a graph. (See~\cite{barioli, BBF2013} and the recent monograph~\cite{ZF-book} surveying inverse problems and zero forcing in graphs.) Zero forcing is defined by the process, which starts by choosing a set $S$ of vertices of a graph $G$ and coloring them blue. The \emph{color-change rule} consists of identifying a blue vertex having only one non-blue neighbor and coloring that neighbor blue. The color-change rule is performed as long as possible. If at the end of the process all vertices become blue, then the initial set $S$ is a \emph{zero forcing set} of $G$. The minimum cardinality of a zero forcing set in $G$ is the \emph{zero forcing number}, $Z(G)$, of $G$.

Several bounds are known for the zero forcing number, where the trivial lower bound $Z(G)\ge \delta(G)$ involves the minimum degree of a graph. Many authors studied upper bounds on $Z(G)$ involving the order and the maximum degree of $G$, and considered also extremal families of graphs attaining the bounds~\cite{acdp-2015,gprs-2016,gms-2018}. For instance, the most recent such bound is $Z(G)\le \frac{n(\Delta - 2)}{\Delta - 1}$, which holds when $G$ is a connected graph with maximum degree $\Delta\ge 3$ and $G$ is not isomorphic to one of the five sporadic graphs~\cite{gr-2018}. Bounds on the zero forcing number expressed in terms of order, maximum and minimum degree were also proved~\cite{cp-2015} and the extremal graphs have recently been characterized~\cite{lx-2023}. We remark that good upper bounds on $Z(G)$ are particularly interesting, because they also present upper bounds on the maximum nullity of a graph.

Domination in graphs is one of the most studied topics in graph theory; see a recent monograph~\cite{HaHeHe-23} surveying core concepts of domination theory.  On the first sight, domination does not seem related to maximum nullity and zero forcing, yet there are some surprising connections. Notably, the concept of power domination, which was introduced in~\cite{haynes-2002} as a model for monitoring electrical networks, has a very similar definition to that of zero forcing. In particular, the corresponding graph invariant $\gamma_P(G)$ is a lower bound for $Z(G)$ in all graphs $G$~\cite{BFF-2018}.  Moreover, the so-called Z-Grundy domination number $\grz(G)$ of $G$, as introduced in~\cite{bbgkkptv-2017}, is dual to the zero forcing number of $G$. In particular, $Z(G)=n(G)-\grz(G)$ holds for every graph $G$ (where $n(G)$ is the order of $G$).

In this paper, we further relate zero forcing with domination. In particular, we prove an upper  bound on the zero forcing number of a graph $G$, expressed in terms of the total domination number $\gamma_t(G)$ of $G$.  (The latter invariant is one of the central concepts of graph domination, cf.~\cite{HaHeHe-23} and the book~\cite{book-total} surveying total domination in graphs). We prove that for any graph $G$ with no component isomorphic to a complete graph, we have $Z(G)\le n(G)-\gamma_t(G)$ and $Z(G)\le n(G)-\frac{\Gamma_t(G)}{2}$ where $\Gamma_t(G)$ is the upper total domination number of $G$. We prove that both of these bounds are sharp, present several properties of the two families of extremal graphs, and prove that any graph is an induced subgraph of a graph $G$ with $Z(G)= n(G)-\frac{\Gamma_t(G)}{2}$. We also show that the ratio $\frac{n-Z(G)}{\Gamma_t(G)}$ can be arbitrarily large (which implies that the ratio $\frac{n-Z(G)}{\gamma_t(G)}$ can also be arbitrarily large). In the proofs, we are often using the language of Z-Grundy domination, which we present in the next section. We are also studying the graphs $G$ with $Z(G)=\delta(G)$, that is, achieving the trivial lower bound.  We prove a characterization of these graphs by widely extending a characterization of the graphs $G$ with zero forcing number equal to $2$ due to Row~\cite{Row-12}. The result is obtained from a characterization of the graphs with power domination equal to $1$, which we also prove.

We denote the \emph{degree} of a vertex $v$ in a graph $G$ by $\deg_G(v)$.  An \emph{isolated vertex} is a vertex of degree~$0$, and an \emph{isolate-free graph} is a graph with no isolated vertices. Hence if $G$ is an isolate-free graph, then $\deg_G(v) \ge 1$ for all vertices $v \in V(G)$. In the following section, we present main definitions used in the paper. In particular, we recall the definition of the Z-Grundy domination number, $\grz(G)$, and the equality $Z(G)=n(G)-\grz(G)$ which holds in all graphs $G$ (note that in the seminal paper~\cite{bbgkkptv-2017} Z-Grundy domination number was introduced for isolate-free graphs, but our definition presented in the next section, allows also isolated vertices). In Section~\ref{sec:totaldom}, we are concerned with total domination, we prove the bound  $\grz(G)\ge \gamma_t(G)$, and present several properties of graphs $G$ that attain the equality $\grz(G)=\gamma_t(G)$. In Section~\ref{sec:uppertotal}, we prove the mentioned results involving the upper total domination number. Finally, in Section~\ref{sec:powerdom}, we prove the characterization of the graphs $G$ with power domination $1$ and the graphs $G$ with $Z(G)=\delta(G)$. Several open problems are posed throughout the paper.

\section{Definitions and notation}
\label{sec:def}

For graph theory notation and terminology, we generally follow~\cite{HaHeHe-23}.  Specifically, let $G$ be a graph with vertex set $V(G)$ and edge set $E(G)$, and of order $n(G) = |V(G)|$ and size $m(G) = |E(G)|$. If $G$ is clear from the context, we simply write $V = V(G)$ and $E = E(G)$. The \emph{open neighborhood} of a vertex $v$ in $G$ is $N_G(v) = \{u \in V \, \colon \, uv \in E\}$ and the \emph{closed neighborhood of $v$} is $N_G[v] = \{v\} \cup N_G(v)$. We denote the degree of $v$ in $G$ by $\deg_G(v)$, and so $\deg_G(v) = |N_G(v)|$. Two vertices are \emph{neighbors} if they are adjacent. For a set $X \subseteq V(G)$ and a vertex $v \in V(G)$, we denote by $\deg_X(v)$ the number of neighbors of $v$ in $G$ that belong to the set $X$, that is, $\deg_X(v) = |N_G(v) \cap X|$. In particular, if $X = V(G)$, then $\deg_X(v) = \deg_G(v)$. The subgraph of $G$ induced by a set $D \subseteq V$ is denoted by $G[D]$. For $k \ge 1$ an integer, we let $[k]$ denote the set $\{1,\ldots,k\}$.

A vertex $x \in V$ \emph{dominates} a vertex $y$ if $y\in N_G[x]$, and we say that $y$ is \emph{dominated} by~$x$. If $D \subseteq V$, then a vertex $y$ in $G$ is \emph{dominated} by $D$ if there exists $x \in D$ that dominates $y$. A set $D$ is a \emph{dominating set} of a graph $G$ if every vertex in $G$ is dominated by $D$.

A vertex $x \in V$ \emph{totally dominates} a vertex $y$ if $y\in N_G(x)$, and we then also say that $y$ is \emph{totally dominated} by $x$. If $D \subseteq V$, then $y \in V$ is \emph{totally dominated} by $D$ if there exists $x\in D$ that totally  dominates $y$. A set $D$ is a \emph{total dominating set} (or shortly, a \emph{TD}-\emph{set}) of $G$ if every vertex in $G$ is totally dominated by $D$. The minimum cardinality of a TD-set of $G$ is the \emph{total domination number} of $G$, denoted $\gamma_t(G)$. A TD-set $D$ in $G$ such that $S$ is not a TD-set of $G$ whenever $S \subsetneq D$, is a \emph{minimal TD-set} of $G$. The maximum cardinality of a minimal TD-set in $G$ is the \emph{upper total domination number}, $\Gamma_t(G)$, of $G$. A $\gamma_t$-\emph{set of}  $G$ is a TD-set of $G$ of cardinality $\gamma_t(G)$, while a $\Gamma_t$-\emph{set of}  $G$ is a minimal TD-set of $G$ of cardinality $\Gamma_t(G)$.

For a set $D \subseteq V$ and a vertex $v \in D$, the \emph{open $D$-private neighborhood} of $v$, denoted by $\pn(v,D)$, is the set of vertices that are in the open neighborhood of $v$ but not in the open neighborhood of the set $D \setminus \{v\}$. Equivalently, $\pn(v,D) = \{w \in V \, \colon N(w) \cap D = \{v\}\}$. The \emph{$D$-external private neighborhood} of $v$ is the set $\epn(v,D) = \pn(v,D) \setminus D$, and its \emph{open $D$-internal private neighborhood} is the set $\ipn(v,D) = \pn(v,D) \cap D$. We note that $\pn(v,D) = \ipn(v,D) \, \cup \, \epn(v,D)$. A vertex in $\epn(v,D)$ is a \emph{$D$-external private neighbor} of $v$, and a vertex in $\ipn(v,D)$ is a \emph{$D$-internal private neighbor} of $v$. Hence, if $w \in \epn(v,D)$, then $w \notin D$ and $w$ is not totally dominated by $D \setminus \{w\}$, while if $w \in \ipn(v,D)$, then $w \in D$ and $w$ is not totally dominated by $D \setminus \{w\}$. A fundamental property of minimal TD-sets (see~\cite[Lemma~4.25]{HaHeHe-23}) is that $D$ is a minimal TD-set in $G$ if and only if $|\epn(v,D)| \ge 1$ or $|\ipn(v,D)| \ge 1$ hold for every vertex $v \in D$.

The concept of Grundy domination can be presented by a sequence of vertices in a graph. The first type of the corresponding graph invariant, the so called Grundy domination number, was defined in~\cite{bgmrr-2014} with a motivation coming from the domination game. A few years latter, the Grundy total domination number~\cite{bhr} and the Z-Grundy domination number~\cite{bbgkkptv-2017} were introduced. A sequence $S=(v_1,\ldots,v_k)$ of vertices in a graph $G$ is a \emph{Z-sequence} if for every $i\in \{2,\ldots,k\}$,

\begin{equation}
\label{eq:defZGrundy}
N_G(v_i) \setminus \bigcup_{j=1}^{i-1}N_G[v_j] \ne \emptyset. \1
\end{equation}
The corresponding set of vertices from the sequence $S$ will be denoted by $\widehat{S}$.
The maximum length $|\widehat{S}|$ of a Z-sequence $S$ in a graph $G$ is the \emph{Z-Grundy domination number}, $\grz(G)$, of $G$. (The definition of \emph{Grundy total domination number} of a graph $G$, $\grt(G)$, is similar,  one just needs to change the closed neighborhood symbol in~\eqref{eq:defZGrundy} with the open neighborhood symbol.) 
If $(v_1,\ldots,v_k)$ is a Z-sequence, then we say that $v_i$ \emph{footprints} the vertices from $N_G[v_i] \setminus \cup_{j=1}^{i-1}N_G[v_j]$, and that $v_i$ is the \emph{footprinter} of every vertex $u\in N_G[v_i] \setminus \cup_{j=1}^{i-1}N_G[v_j]$, for any $i\in [k]$ (where $[k]=\{1,\ldots,k\}$).
Note that if $S$ is a Z-sequence, $x\in \widehat{S}$ may footprint itself, but it must footprint (also) a vertex $y$ distinct from $x$. For a sequence $S=(u_1,\ldots , u_k)$ and a vertex $x \notin \widehat{S}$, we use the following notation: $S\oplus (x)=(u_1,\ldots , u_k,x)$.

As it turns out, the Z-Grundy domination number is dual to the zero forcing number. Moreover, a sequence $S$ is a Z-sequence if and only if the set of vertices outside $S$ forms a zero forcing set~\cite{bbgkkptv-2017}. In particular, this implies that
\begin{equation}
\label{eq:zeroZGrundy}
Z(G)=n(G)-\grz(G)
\end{equation}
for every graph $G$. In a subsequent paper, Lin presented a natural connection between four variants of Grundy domination and four variants of zero forcing~\cite{lin-2019}.  The connections show that all versions of Grundy domination can be applied in the study of different types of minimum rank parameters of symmetric matrices.

Given a graph $G$, a complete subgraph is a \emph{clique} in $G$.
Similarly, a complete graph may also be called a clique. A vertex of degree~$1$ in $G$ is called a \emph{leaf} of $G$. A vertex $v \in V$ is a \emph{simplicial vertex} of $G$, if $N_G(v)$ induces a clique. The graph $G$ is \emph{chordal} if it contains no induced cycles of length more than~$3$. Two vertices $u$ and $v$ in $G$ are \emph{closed twins} if $N_G[u] = N_G[v]$ and \emph{open twins} if $N_G(u)=N_G(v)$. Vertices $u$ and $v$ in $G$ are \emph{twins} if they are either open or closed twins. A vertex $v \in V$ is called a \emph{twin} vertex if there exists $u \in V(G)$, such that $u$ and $v$ are twins.

In all notations presented in this section, index $G$ may be omitted if the graph $G$ is understood from the context.

\section{Zero forcing and total domination}
\label{sec:totaldom}

If $G$ is an isolate-free graph and $D$ is a (minimal) TD-set of $G$, then the induced subgraph $G[D]$ is isolate-free. In addition, the following observation is easy to see.

\begin{ob}
\label{ob:private}
Let $G$ be an isolate-free graph and let $D$ be a (minimal) TD-set of $G$.
If $x$ belongs to a component $C$ of $G[D]$ such that $x$ is not adjacent to a vertex $y\in V(C)$ with $\deg_C(y)=1$, then $x$ has an external private neighbor with respect to $D$.
\end{ob}

We will also make use of the following notation. Let $D$ be a (minimal) TD-set of $G$, and let $C_1,\ldots, C_{\ell}$ be the components of $G[D]$, which are isomorphic to $K_2$; possibly $\ell=0$ when there are no such components.
For each $i\in [\ell]$, let $A_i(D)$ denote the set of vertices that are totally dominated by $V(C_i)$ and are not totally dominated by $D\setminus V(C_i)$. In particular, $V(C_i)\subset A_i(D)$ for all $i\in [\ell]$, since both vertices in $C_i$ are (internal) private neighbors to each other.

\begin{thm}
\label{thm:gammaZ-total}
If $G$ is a graph such that no component of $G$ is a clique, then $\grz(G) \ge \gamma_t(G)$.
\end{thm}
\begin{proof}
In the proof, we will construct a Z-sequence $S$ of $G$, with $|\widehat{S}|\ge |D|$, where $D$ is a $\gamma_t$-set of $G$. Among all $\gamma_t$-sets of $G$, let $D$ be chosen in such a way that $G[D]$ has the smallest possible number of $K_2$-components.
In addition, letting $C_1,\ldots, C_{\ell}$ be the $K_2$-components of $G[D]$, and $V(C_i)=\{x_i,y_i\}$ for all $i\in [\ell]$, we choose $D$ among all $\gamma_t$-set of $G$ restricted as above in such a way that the number of $K_2$-components in $G[D]$ for which $(N(x_i)\cap A_i(D))\setminus \{y_i\}=(N(y_i)\cap A_i(D))\setminus \{x_i\}$ is as small as possible.

Since $D$ is fixed, we may simplify the notation and write $A_i$ instead of $A_i(D)$ for the vertices that are totally dominated by $x_i$ or $y_i$, and are not adjacent to any vertex in $D\setminus V(C_i)$.

First, consider the components of $G[D]$ that are not isomorphic to $K_2$. For each such component $C$, first add to $S$ all vertices $u$ in $G[D]$ such that $u$ is  adjacent to a vertex $v\in V(C)$ with $\deg_C(v)=1$. Each such vertex $u$ footprints a corresponding vertex $v$ noting that $v$ is a $D$-internal private neighbor of $u$. Thereafter, we add to $S$ all the remaining vertices in $C$. By Observation~\ref{ob:private}, every such vertex $x$, since it has no $D$-internal private neighbor, has a $D$-external private neighbor $y$. Therefore, $x$ footprints $y$, and so the resulting sequence $S$, constructed so far, is a Z-sequence. Note that the number of vertices added to $S$ from $C$ is $|V(C)|$ since all vertices of $C$ have been added to $S$.

Dealing in the same way with all components $C$ of $G[D]$, where $|V(C)|\ge 3$, as explained in the previous paragraph, we are left with the components $C_1,\ldots, C_{\ell}$ of $G[D]$, where $V(C_i)=\{x_i,y_i\}$ for all $i\in [\ell]$. Without loss of generality, we may chose the indices of the components $C_i$ in such a way that for $C_i$, where $i\in [k]$, we have
\begin{equation}
\label{eq:property2}
(N(x_i)\cap A_i)\setminus \{y_i\} = (N(y_i)\cap A_i)\setminus \{x_i\},
\end{equation}
while for $i\in\{k+1,\ldots,\ell\}$, components $C_i$ do not have this property, that is,
$(N(x_i)\cap A_i)\setminus \{y_i\} \ne (N(y_i)\cap A_i)\setminus \{x_i\}$. Possibly $k=0$ or $k=\ell$, which represent cases where only one of the two types of components appears. Renaming vertices if necessary, we may assume without loss of generality that
\begin{equation}
\label{eq:property3}
\bigl(N(x_i)\setminus N[y_i]\bigr)\cap A_i\ne \emptyset,
\end{equation}
and let $a_i\in \bigl(N(x_i)\setminus N[y_i]\bigr)\cap A_i$. For each $i\in \{k+1,\ldots,\ell\}$, we first add to $S$ the vertex $y_i$, and then the vertex $x_i$. In this way, $y_i$ footprints $x_i$, while $x_i$ footprints $a_i$ and thus the so extended sequence $S$ is still a Z-sequence. We note that we have added two vertices to $S$ from each component $C_i$, where $|V(C_i)|=2$.

Finally, we deal with components $C_i$, where $i\in [k]$. We claim that none of the sets $A_i$ induces a complete graph. Suppose, to the contrary, that $G[A_i]$ is a clique. Since no component of $G$ is a clique, there exists a vertex $a\in A_i$, which is adjacent to a vertex $b\notin A_i$. We note that $b\notin D$, and by the definition of sets $A_i$, the vertex $b$ is totally dominated by a vertex $c\in D\setminus V(C_i)$. Now, $D'=(D\setminus\{x_i,y_i\})\cup\{a,b\}$ is a TD-set of $G$. Since $|D'| = |D| = \gamma_t(G)$, the TD-set $D'$ is therefore a $\gamma_t$-set of $G$. Further vertex $a$ is a $D'$-internal private neighbor of vertex $b$ and vertices in $A_i$ are $D'$-external private neighbors of vertex~$a$. Now, $G[D']$ has fewer $K_2$-components as $G[D]$, which is a contradiction to the initial assumption on $D$. Therefore, no set $A_i$ induces a clique for $i \in [k]$.

For each $i\in [k]$, there therefore exists a vertex $a_i \in A_i \setminus \{x_i,y_i\}$ such that $a_i$ is not adjacent to a vertex $b_i \in A_i\setminus\{x_i,y_i\}$. Now, $D'=(D\setminus\{y_1\})\cup \{a_1\}$ is a TD-set of $G$, since vertices in $V(G)\setminus A_1$ are totally dominated by $D\setminus V(C_1)$, vertices of $A_1\setminus\{x_1\}$ are totally dominated by $x_1$, and $x_1$ is totally dominated by $a_1$. Further, $G[D']$ has the same number of $K_2$-components as $G[D]$. However, since
\[
(N(x_1)\cap A_1(D'))\setminus \{a_1\} \ne (N(a_1)\cap A_1(D'))\setminus \{x_1\},
\]
the number of components with the additional property given in Equation~(\ref{eq:property2}) is fewer in $D'$ than in $D$. Therefore, we are in contradiction with the initial assumption on $D$, and so this case does not appear, implying that $k=0$. We infer that $\widehat{S}=D$, and so $S$ is a Z-sequence of length $\gamma_t(G)$, which yields $\grz(G)\ge \gamma_t(G)$.~\QED
\end{proof}

\bigskip

Translated to the zero forcing number, Equation~(\ref{eq:zeroZGrundy}) and Theorem~\ref{thm:gammaZ-total} imply the following bound on the zero forcing number, where a \emph{clique component} of a graph is a component of the graph that is a clique.

\begin{cor}
\label{cor:ZF-total}
If $G$ is a graph with no clique component, then $Z(G)\le n(G)-\gamma_t(G)$.
\end{cor}

 The difference between the total domination number and the Z-Grundy domination number of a graph can be arbitrary large. Moreover, the ratio $\frac{\grz(G)}{\gamma_t(G)}$ can be arbitrarily large. For instance, let $G_k$, where $k\in \mathbb{N},$ be the graph obtained from the disjoint union of two copies of the complete graph $K_k$ by adding edges that form a perfect matching of $G_k$. Clearly, $\grz(G_k)=k$ and $\gamma_t(G_k)=2$.

On the other hand, the bound in Theorem~\ref{thm:gammaZ-total} is sharp. For instance, for the star $K_{1,\ell}$ we have $\grz(K_{1,\ell})=2=\gamma_t(G)$. In the remainder of this section, we study properties of graphs whose Z-Grundy domination number equals the total domination number.

Since $\gamma_t(G) \le \grz(G) \le \grt(G)$ holds for any graph $G$ with no component isomorphic to complete graph, any graph $G$ with $\gamma_t(G)=\grt(G)$ also satisfies $\gamma_t(G)=\grz(G)$. Graphs with $\gamma_t(G)=\grt(G)=k$, which are known under the name \emph{total $k$-uniform graphs}, were studied in~\cite{bahadir-2020,bhr,DJ-21}. In these papers, bipartite graphs with $\gamma_t(G)=\grt(G)=4$ and with $\gamma_t(G)=\grt(G)=6$ were characterized. It was also shown that all connected total $k$-uniform graphs having no open twins are regular, and some examples of non-bipartite total $k$-uniform graphs, for even $k$, were also presented. Moreover, chordal total $k$-uniform graphs were characterized. (In particular, for $k \ge 3$ no such graphs exist). By the above observation, all total $k$-uniform graphs are also graphs with $\gamma_t(G)=\grz(G)$, but this property holds also for many other graphs. For instance, while complete multipartite graphs are the only connected totally 2-uniform graphs (see~\cite{bhr}), we have the following characterization of connected graphs with $\gamma_t(G)=\grz(G)=2$.

\begin{prop}\label{prp:2}
Let $G$ be a connected graph not isomorphic to a complete graph. Then $\gamma_t(G)=\grz(G)=2$ if and only if $N[x] \cup N[y]=V(G)$ holds for any non-twin vertices $x$ and $y$.
\end{prop}
\begin{proof}
First let $N[x] \cup N[y]=V(G)$ hold for any non-twin vertices $x$ and $y$. Suppose that there is a $Z$-sequence $(u,v,w)$ in $G$, implying that $N(w)\setminus (N[u] \cup N[v]) \ne \emptyset$. Since $N[x] \cup N[y]=V(G)$ holds for any non-twin vertices $x$ and $y$ of $G$, $u$ and $v$ must be twins. Hence $N(v) \setminus N[u]=\emptyset$, a contradiction. Hence, $\grz(G) \le 2$ and since $G$ is not complete, $\grz(G)=2$. By Theorem~\ref{thm:gammaZ-total}, $\gamma_t(G)=2$.

For the converse, let $\gamma_t(G)=\grz(G)=2$. Suppose that there exists non-twin vertices $u$ and $v$ with $N[u] \cup N[v] \ne V(G)$. Since $G$ is connected, there exists $w \in N(u) \cup N(v)$ such that $w$ has a neighbor $w' \in V(G) \setminus (N[u] \cup N[v])$. Then $(u,v,w)$ or $(v,u,w)$ is a Z-sequence of $G$ and hence $\grz(G)\ge 3$, a contradiction. \QED
\end{proof}
\bigskip

By the results proved in~\cite{bahadir-2020} we know that there are no total $k$-uniform graphs for odd $k$. We can easily find graphs with $\gamma_t(G)=\grz(G)=k$ for some odd~$k$. For example, $C_5$ has both Z-Grundy domination number and total domination number equal to~$3$. We prove next that there are no chordal graphs with $\gamma_t(G)=\grz(G)=3$.
To prove this we first need the following observation, where statement~(a) was proved in~\cite{kos}, while statement (b) is straightforward to verify.

\begin{ob}
\label{ob:simplicial}
If $G$ is a graph and $u$ a simplicial vertex of $G$, then \\ [-20pt]
\begin{enumerate}
\item[{\rm (a)}] $\grz(G)-1 \le \grz(G-u) \le \grz(G)$, and \1
\item[{\rm (b)}] $\gamma_t(G)-1 \le \gamma_t(G-u) \le \gamma_t(G)$.
\end{enumerate}
\end{ob}

\begin{thm}
\label{thm:chordal}
If $G$ is a connected graph that contains a simplicial vertex, then $\gamma_t(G) \ne 3$ or $\grz(G) \ne 3$.
\end{thm}
\begin{proof}
Suppose, to the contrary, that there are connected graphs having simplicial vertices with both Z-Grundy domination number and total domination number equal to~$3$. Among all such graphs $G$ that contain a simplicial vertex and satisfy $\gamma_t(G)=\grz(G)=3$, let $G$ be chosen to have minimum order. Clearly, $G$ is not a complete graph.  Let $x$ be a simplicial vertex of $G$, and let $X=N[x]$, and so $X$ induces a clique. Let $x_1,x_2$ be two arbitrary vertices from $X \setminus \{x\}$.

If $(x,x_1,x_2)$ is a Z-sequence of $G$, then since $\grz(G)=3$ we infer that $N[x] \cup N[x_1] \cup N[x_2] = V(G)$, implying that $\{x_1,x_2\}$ is a TD-set of $G$, a contradiction. Hence, $(x,x_1,x_2)$ is not a Z-sequence of $G$. Since $x_1$ and $x_2$ were arbitrary, neither $(x,x_1,x_2)$ nor $(x,x_2,x_1)$ is a Z-sequence of $G$. Thus since $X$ induces a clique, we infer that $N[x_1]=N[x]$ or $N[x_1]=N[x_2]$. We can therefore partition $X$ into two sets, $A=\{a \in X \, \colon N[a]=N[x]\}$ and $B=X\setminus A$. Since $G$ is not complete, the set $B$ is not empty. If $|B| \ge 2$, then choosing $\{x_1,x_2\} \subseteq B$, the sequence $(x,x_2,x_1)$ is not a Z-sequence of $G$, implying that $N[x_1]=N[x_2]$, and so $x_1$ and $x_2$ are twins in $G$. Thus if $|B| \ge 2$, then any two vertices of $B$ are twins in $G$. Let $b \in B$ and let $C = N(b) \setminus N[x]$, which is clearly non-empty. By our earlier observations, $X = A \cup B$ is a clique, $N[a] = X$ for all $a \in A$, and $N[b] = A \cup B \cup C$ for all $b \in B$.

Let $H = G-x$. By our earlier properties, the graph $H$ is a connected isolate-free graph. Further, $H$ is not a complete graph. Hence by Theorem~\ref{thm:gammaZ-total}, $\grz(H) \ge \gamma_t(H) \ge 2$. By Observation~\ref{ob:simplicial}, we have $\gamma_t(H) \le 3$ and $\grz(H) \le 3$. If $\gamma_t(H) = 3$, then $\grz(H) = 3$, contradicting the minimality of $G$. Hence, $\gamma_t(H) = 2$ and either $\grz(H) = 2$ or $\grz(H) = 3$. Let $\{u_1,u_2\}$ be an arbitrary $\gamma_t$-set of $H$. If $\{u_1,u_2\} \cap X \ne \emptyset$, then $\{u_1,u_2\}$ is a TD-set of $G$, and so $\gamma_t(G) = 2$, a contradiction. Hence, every $\gamma_t$-set of $H$ has an empty intersection with $X$. Therefore, $A = \{x\}$, and $\{b,c\}$ is not a TD-set of $H$ for any $b \in B$ and $c \in C$. We now select $b \in B$. Since $N[b] = B \cup C$ for all $b \in B$ and since $\{b,c\}$ is not a TD-set of $H$ for any $b\in B$ and $c\in C$, we can select $c \in C$ in such a way that $D = N[c] \setminus (B \cup C)$ is non-empty. In addition, there exists vertices $d\in D$ and $e\notin B\cup C\cup D$ such that $e$ is adjacent to $d$ or to a vertex $c'\in C$. Therefore, $(x,b,c,d)$ or $(x,b,c,c')$ is a Z-sequence, where vertex $x$ footprints $b$, vertex $b$ footprints $c$, vertex $c$ footprints $d$, and, finally, vertex $d$ or vertex $c'$ footprints $e$. Hence, $\grz(G) \ge 4$, a contradiction. \QED
\end{proof}

\medskip
Since every chordal graph which is not complete has a simplicial vertex, as an immediate consequence of Theorem~\ref{thm:chordal} we have the following result.

\begin{cor}
There is no connected chordal graph $G$ such that $\gamma_t(G) = \grz(G) = 3$.
\end{cor}

We already know~\cite{bahadir-2020} that there are no total $k$-uniform chordal graphs for $k \ge 3$. This raises the following question.

\begin{prob}
Is there a connected chordal graph $G$ with $\gamma_t(G)=\grz(G)=k$ for $k \ge 4$?
\end{prob}

We end the section with the following problem that arises from the above discussion.

\begin{prob}
\label{prob:gammaZ-total}
Characterize the extremal graphs attaining the bound in Theorem~\ref{thm:gammaZ-total}.
\end{prob}

From earlier papers one can suspect that finding a characterization of graphs $G$ with $\grt(G)=\gamma_t(G)$ should be difficult. Thus we expect that Problem~\ref{prob:gammaZ-total} will also not be easy.

\section{Zero forcing and upper total domination}
\label{sec:uppertotal}

Comparing the Z-Grundy domination number with the upper total domination number of a graph, we note that the inequality in Theorem~\ref{thm:gammaZ-total} cannot be improved by replacing $\gamma_t$ with $\Gamma_t$. In other words, there exist graphs such that the upper total domination number is larger than the Z-Grundy domination number, and the difference can even be arbitrarily large. For instance, the \emph{windmill graph} $\Wd(k,n)$, where $k \ge 3$ and $n\ge 2$, is obtained by taking $n$ vertex disjoint copies of the complete graph $K_k$, selecting one vertex from each copy, and identifying these $n$ selected vertices into one new vertex (that is a universal vertex of degree~$n(k-1)$). The resulting graph windmill graph $G = \Wd(k,n)$ satisfies $\grz(G)=n$ and $\Gamma_t(G)=2n$. This yields the following result.

\begin{ob}
\label{ob:ratio=1}
There exists an infinite family of connected isolate-free graphs $G$ satisfying $\Gamma_t(G) = 2\grz(G)$.
\end{ob}

By Observation~\ref{ob:ratio=1}, if $G$ is a connected isolate-free graph, then the ratio $\frac{\Gamma_t(G)}{\grz(G)}$ can be as large as $2$. However, we next prove that this ratio cannot exceed~$2$.

\begin{thm}
\label{thm:gammaZ-uppertotal}
If $G$ is an isolate-free graph, then $\Gamma_t(G) \le 2\grz(G)$.
\end{thm}
\begin{proof}
We will use (a simplified version of) the setting described in the proof of Theorem~\ref{thm:gammaZ-total}. Let $D$ be a $\Gamma_t$-set of $G$, and so $D$ is a minimal TD-set of $G$ of cardinality $\Gamma_t(G)$, and consider the components of the induced subgraph $G[D]$. We will construct a Z-sequence $S$ of $G$, with $|\widehat{S}| \ge |D|/2$.

Let ${\cal C}_1$ be the set of all components of $G[D]$ with more than two vertices, and let ${\cal C}_2$ be the set of all $K_2$-components of $G[D]$. Note that $\sum_{C\in {\cal C}_1\cup {\cal C}_2 }{|V(C)|}=|D|$.

If $C\in {\cal C}_1$ then, in the same way as in the proof of Theorem~\ref{thm:gammaZ-total}, we can add $|V(C)|$ vertices to $S$ (by first adding to $S$ all vertices of $C$ that are neighbors of leaves of $C$, and then all other vertices of $C$). After dealing in this way with all components from ${\cal C}_1$, we focus on the remaining components of $G[D]$ (having two vertices). From each $C\in {\cal C}_2$, we can add at least one vertex to the sequence $S$ in such a way that the added vertex footprints its neighbor in $D$. The resulting sequence $S$ is a Z-sequence, and so $\grz(G)\ge |\widehat{S}|$. On the other hand, \[
|\widehat{S}| = \sum_{C\in {\cal C}_1}{|V(C)|} + \sum_{C\in {\cal C}_2}{\frac{|V(C)|}{2}} \ge \sum_{C\in {\cal C}_1\cup {\cal C}_2 }{\frac{|V(C)|}{2}} = \frac{|D|}{2} = \frac{\Gamma_t(G)}{2}.
\]
This implies that $\grz(G)\ge \frac{1}{2}\Gamma_t(G)$, as claimed.
\QED
\end{proof}

\bigskip
As an immediate consequence of Corollary~\ref{thm:gammaZ-uppertotal}, we have the following corollary.

\begin{cor}
\label{cor:ZF-uppertotal}
If $G$ is an isolate-free graph, then $Z(G)\le n(G)-\frac{\Gamma_t(G)}{2}$.
\end{cor}

By the example leading to Observation~\ref{ob:ratio=1}, the windmill graphs $G = \Wd(k,n)$ where $k \ge 3$ and $n\ge 2$ attain the upper bound in Theorem~\ref{thm:gammaZ-uppertotal}. Next, we present some properties of graphs $G$ with $\Gamma_t(G) = 2\grz(G)$.  We will again use the following notation.

Let $D$ be a $\Gamma_t$-set of an isolate-free graph $G$, and let $C_1,\ldots,C_\ell$ be the $K_2$-components of $G[D]$ where $V(C_i) = \{x_i,y_i\}$ for $i \in [\ell]$. For each $i\in [\ell]$, let $A_i(D)$ denote the set of vertices that are totally dominated by $V(C_i)$ and are not totally dominated by $D\setminus V(C_i)$.  In particular, $V(C_i) \subseteq A_i(D)$ for all $i\in [\ell]$, since both vertices in $V(C_i)$ are $D$-internal private neighbors to each other. From the proof of Theorem~\ref{thm:gammaZ-total} and Theorem~\ref{thm:gammaZ-uppertotal}, we immediately get the following.

\begin{lemma}
\label{l:K2}
If $G$ is a graph with  $\Gamma_t(G)=2\grz(G)$ and $D$ is a $\Gamma_t$-set of $G$, then each component of $G[D]$ is isomorphic to $K_2$ and $N[x_i] \cap A_i(D)=N[y_i] \cap A_i(D)$ for all $i\in[\ell]$.
\end{lemma}

We note that Lemma~\ref{l:K2} implies that $V(C_1) \cup \cdots \cup V(C_\ell)=D$.

\begin{lemma}
\label{l:cliques}
If $G$ is a graph with  $\Gamma_t(G)=2\grz(G)$ and $D$ is a $\Gamma_t$-set of $G$, then the following properties hold. \\ [-20pt]
\begin{enumerate}
\item[{\rm (a)}] $A_i(D)$ induces a clique for all $i \in [\ell]$, and
\item[{\rm (b)}] there are no edges between $A_i(D)$ and $A_j(D)$ for all $i$ and $j$ where $1 \le i < j \le \ell$.
\end{enumerate}

\end{lemma}
\begin{proof}
Let $G$ be a graph with  $\Gamma_t(G) = 2\grz(G)$ and let $D$ be a $\Gamma_t$-set of $G$. It follows from Lemma~\ref{l:K2} that each component $C_i$ of $G[D]$ is a $K_2$-component, implying that $\Gamma_t(G)=2\ell$. Further, $N[x_i] \cap A_i(D)=N[y_i] \cap A_i(D)$ for all $i \in [\ell]$. Suppose that there exist $a,b \in A_i(D)$ such that $ab \notin E(G)$. Renaming components if necessary, we may assume that $i = 1$. Thus, $(a,x_1,x_2,\ldots , x_\ell)$  is a Z-sequence of $G$, since vertex $a$ footprints $x_1$, vertex $x_1$ footprints $b$, and vertex $x_j$ footprints $y_j$ for all $j \ge 2$. Thus, $\grz(G) \ge \ell + 1$, and so, $\gamma_t(G) = 2\ell < 2\grz(G)$, a contradiction. Hence, $A_i(D)$ induces a clique for all $i \in [\ell]$. This proves part~(a).

To prove part~(b), suppose that there exists an edge $e=ab$ between $A_i(D)$ and $A_j(D)$ for some $i$ and $j$ where $1 \le i < j \le \ell$, where $a \in A_i(D)$ and $b \in A_j(D)$. Since $A_i(D)$ is, by definition, the set of neighbors of $\{x_i,y_i\}$ that are not totally dominated by $D\setminus \{x_i,y_i\}$, no vertex from the set $\{x_i,y_i,x_j,y_j\}$ is incident with the edge~$e$. Hence, $(x_1,\ldots ,x_i,a, x_{i+1},\ldots , x_\ell)$ is a Z-sequence of $G$, since each vertex $x_p$ footprints $y_p$ for $p \in [\ell]$ and vertex $a$ footprints vertex $b$. Thus, $\grz(G) \ge \ell+1$, a contradiction.
\QED
\end{proof}

\bigskip
In the rest of this section we will denote by $H$ the subgraph $G-(A_1(D) \cup \cdots \cup A_\ell(D))$. If there is a vertex $v \in V(G)$ such that $v$ is adjacent to all vertices of a set $X \subset V(G)$, then we will use the notation $v \sim X$.

\begin{lemma}\label{l:twins}
If $G$ is a graph with  $\Gamma_t(G)=2\grz(G)$ and $D$ is a $\Gamma_t$-set of $G$, then the vertices in $A_i(D)$ are closed twins for all $i \in [\ell]$.
\end{lemma}
\begin{proof}
Let $G$ be a graph with  $\Gamma_t(G) = 2\grz(G) = 2\ell$ and let $D$ be a $\Gamma_t$-set of $G$. Suppose that there exist $a,b \in A_i(D)$, for some $i \in [\ell]$, such that $N[a] \ne N[b]$. Renaming components if necessary, we may assume that $i = 1$. By Lemma~\ref{l:cliques}, the set $A_1(D)$ induces a clique. Hence renaming the vertices $a$ and $b$ if necessary, we may assume without loss of generality that there exists a vertex $u \notin A_1(D)$ such that $u \in N[a] \setminus N[b]$. By Lemma~\ref{l:cliques}, we also infer that $u \in V(H)$. If $\ell = 1$, then $(b,a)$ is a Z-sequence in $G$ noting that vertex~$b$ footprints $A_1(D)$ and vertex~$a$ footprints $u$, implying that $\grz(G) \ge 2 = \ell + 1$, a contradiction. Hence, $\ell \ge 2$. In this case, $(b,a,x_2,\ldots, x_\ell)$ is a Z-sequence in $G$, noting that vertex~$b$ footprints $A_1(D)$, vertex~$a$ footprints $u$, and vertex $x_p$ footprints $A_p(D)$ for all $p \in [\ell] \setminus \{1\}$. This produces a Z-sequence of length $\ell +1$, implying that $\grz(G) \ge \ell+1$, a contradiction. \QED
\end{proof}

\bigskip
We note that by Lemma~\ref{l:twins} and by definition of the sets $A_i(D)$, for every vertex $u \in V(H)$ there exist two distinct indices in $[\ell]$, say $i, j,$ such that $u \sim A_i(D)$ and $u \sim A_j(D)$.

\begin{lemma}\label{l:adjacent}
If $G$ is a graph with  $\Gamma_t(G)=2\grz(G)$ and $D$ is a $\Gamma_t$-set of $G$, then for any adjacent vertices $u,v \in V(H)$ it holds that $|\{i \, \colon \, u\sim A_i(D) \textrm{ and } v\sim A_i(D)\}| \ge 1$.
\end{lemma}
\begin{proof}
Let $G$ be a graph with $\Gamma_t(G)=2\grz(G)$ and let $D$ be a $\Gamma_t$-set of $G$. Let $R = V(G) \setminus V(H)$. Suppose that there exist adjacent vertices $u,v \in V(H)$ with $N_R(u) \cap N_R(v) = \emptyset$. Let $i_1,\ldots , i_k$ be the indices for which $u \sim A_i(D)$ holds and let $\{j_1,\ldots , j_{\ell-k}\} = [\ell] \setminus \{i_1,\ldots , i_k\}$. By supposition the vertex $v$ has no neighbors in $A_{i_1}(D) \cup \cdots \cup A_{i_k}(D)$. Thus, $(x_{i_1},\ldots , x_{i_k},u,x_{j_1},\ldots , x_{j_{\ell-k}})$ is a Z-sequence of $G$ of cardinality $\ell+1$, implying that $\grz(G) \ge \ell+1$, a contradiction. \QED
\end{proof}

\begin{lemma}\label{l:non-adjacent}
If $G$ is a graph with  $\Gamma_t(G)=2\grz(G)$ and $D$ is a $\Gamma_t$-set of $G$, then for any non-adjacent vertices $u,v \in V(H)$ it holds that $|\{i \, \colon \, u\sim A_i(D) \textrm{ and } v\sim A_i(D)\}| \ne 1$.
\end{lemma}
\begin{proof}
Let $G$ be a graph with  $\Gamma_t(G)=2\grz(G)$ and let $D$ be a $\Gamma_t$-set of $G$. Suppose that there exist non-adjacent vertices $u,v \in V(H)$ such that the only common neighbors of $u$ and $v$ in $V(G) \setminus V(H)$ are vertices of $A_{i_1}(D)$ for some $i_1 \in [\ell]$. Let $i_1,\ldots , i_k$ be the indices for which $u \sim A_i(D)$ holds and let $\{j_1,\ldots , j_{\ell-k}\}=[\ell] \setminus \{i_1,\ldots , i_k\}$. By supposition, the vertex $v$ has no neighbors in $A_{i_2} \cup \cdots \cup A_{i_k}$. Then $(x_{i_2},x_{i_3},\ldots , x_{i_k},u,x_{i_1},x_{j_1},\ldots , x_{j_{\ell-k}})$ is a Z-sequence. Indeed, for $i \ne i_1$, the vertex $x_i$ footprints $A_i(D)$, the vertex $u$ footprints $A_{i_1}(D)$, and the vertex $x_{i_1}$ footprints $v$. Thus, $\grz(G) \ge \ell +1$, a contradiction. \QED
\end{proof}

\bigskip
Let $G$ be a graph with $\Gamma_t(G)=2\grz(G)$ and let $D$ be a $\Gamma_t$-set of $G$. For a subset $B \subseteq V(H)$, we will denote by $X_{B}(D)$ the set
\[
X_B(D)=\{i\in[\ell] \, \colon \, u\sim A_i(D) \textrm{ for some } u\in B\}.
\]

We will denote by $m_{B}(D)$ the largest cardinality of a minimal subset $X \subseteq X_B(D)$ such that
\[
B \subseteq \bigcup_{i \in X} N(x_i).
\]
With this notation, we introduce the following lemma.

\begin{lemma}
\label{l:isolated}
If $G$ is a graph with $\Gamma_t(G)=2\grz(G)$ and $D$ is a $\Gamma_t$-set of $G$, then for every $B \subseteq V(H)$ we have
\[
\grz(G[B\setminus B'])+m_{B'}(D)\le |X_B(D)|,
\]
where $B'=\{u\in B \, \colon \, u\textrm{ is an isolated vertex of }G[B]\}$.
\end{lemma}
\begin{proof}
Let $G$ be a graph with $\Gamma_t(G)=2\grz(G)=2\ell$ and let $D$ be a $\Gamma_t$-set of $G$. Let $B \subseteq V(H)$ and $B'=\{u\in B \, \colon \, u\textrm{ is an isolated vertex of }G[B]\}$. Renaming vertices in $D$ if necessary, we may assume without loss of generality that $X_{B}(D) = [k]$. Suppose, to the contrary, that $\grz(G[B\setminus B']) + m_{B'}(D) > k$. Let $S=(a_1,\ldots,a_s)$ be a $\grz$-sequence of $G[B\setminus B']$ and let $X=\{i_1,\ldots,i_m\}\subseteq X_{B'}(D)$ be a minimal subset of cardinality $m=m_{B'}(D)$ such that the vertices in $\{A_i(D) \, \colon \, i\in X\}$ dominate $B'$. We now consider the sequence given by
\[
S'=(a_1,\ldots,a_s,x_{i_1},\ldots,x_{i_m},x_{k+1},\ldots,x_{\ell}).
\]

The sequence $S'$ is a $Z$-sequence of $G$. To see this, note that the vertex $a_i$ footprints some vertex of $G[B\setminus B']$ for all $i \in [s]$ since $S$ is a $\grz$-sequence of $G[B\setminus B']$. By our choice of $X$, the vertex $x_{i_j}$ footprints some $u\in B'$. Finally, since $B \cap N(x_i)=\emptyset$, the vertex $x_{i}$ footprints the vertices in $A_i(D)$ for all $i\ge k+1$. Thus, $S'$ is a $Z$-sequence of $G$ of length
\[
\underbrace{\grz(G[B\setminus B'])+m_{B'}(D)}_{>k}+(\ell-k)>\ell=\grz(G),
\]
and we arise to a contradiction. \QED
\end{proof}

\bigskip
We summarize the above lemmas into the following result in which we adopt the notation established in this section.

\begin{prop}
\label{prp:Gamma=2ZG}
If $G$ is a graph with  $\Gamma_t(G)=2\grz(G)$ and $D$ is a $\Gamma_t$-set of $G$, then the following properties hold. \\ [-20pt]
\begin{enumerate}
\item[{\rm (i)}] each component of $G[D]$ is isomorphic to $K_2$ and so $|D|=2\ell$, for some integer $\ell$; \1
\item[{\rm (ii)}] $N[x_i] \cap A_i(D)=N[y_i] \cap A_i(D)$, for each $i\in[\ell]$; \1
\item[{\rm (iii)}] $A_i(D)$ induces a clique, for each $i \in [\ell]$; \1
\item[{\rm (iv)}] there are no edges between $A_i(D)$ and $A_j(D)$, for each $\{i,j\}\subset [\ell]$;  \1
\item[{\rm (v)}] vertices in $A_i(D)$ are closed twins, for each $i \in [\ell]$; \1
\item[{\rm (vi)}] for every adjacent vertices $u,v \in V(H)$, we have
\[
|\{i:\, u\sim A_i(D) \textrm{ and } v\sim A_i(D)\}| \ge 1;
\]
\item[{\rm (vii)}] for every non-adjacent vertices $u,v \in V(H)$ it holds that
\[
|\{i:\, u\sim A_i(D) \textrm{ and } v\sim A_i(D)\}| \ne 1;
\]
\item[{\rm (viii)}]  for every $B\subset V(H)$, if $B'=\{u\in B:u\textrm{ is an isolated vertex of }G[B]\}$, we have
\[
\grz(G[B\setminus B'])+m_{B'}(D)\le |X_B(D)|.
\]
\end{enumerate}
\end{prop}

We do not think that the combination of properties in Proposition~\ref{prp:Gamma=2ZG} is sufficient for an isolate-free connected graph $G$ to satisfy $\Gamma_t(G)=2\grz(G)$. Therefore, we pose the following problem.

\begin{prob}
\label{prob:2}
Determine if the properties in Proposition~\ref{prp:Gamma=2ZG} are sufficient for an isolate-free connected graph $G$ to satisfy $\Gamma_t(G)=2\grz(G)$. If not, then extend these properties to obtain a characterization of graphs $G$ with $\Gamma_t(G)=2\grz(G)$.
\end{prob}

Now, we present a large family of graphs achieving the bound from Theorem~\ref{thm:gammaZ-uppertotal}. Let $H$ be an arbitrary graph with $\grz(H) \le \ell$. Let $G$ be the graph obtained from $H$ by adding $\ell$ disjoint cliques $K_{n_1},\ldots , K_{n_\ell}$, each of order at least~$2$, and then connecting by an edge every $u \in V(H)$ with every vertex of $K_{n_1} \cup \cdots \cup K_{n_\ell}$. The resulting graph $G$ satisfies $\Gamma_t(G)=2\ell$ and $\grz(G)=\ell$. Since $H$ is an induced subgraph of $G$, we can formulate this observation as follows.

\begin{ob}
Every graph $H$ is an induced subgraph of a graph $G$ with ${\Gamma_t(G)}=2\grz(G)$.
\end{ob}

We note that in many graphs the Z-Grundy domination number is (much) bigger than the upper total domination number, and the ratio ${\grz(G)}/{\Gamma_t(G)}$ can be arbitrarily large. Indeed, let $G$ be an arbitrary graph, and let $G^*$ be obtained from the disjoint union of $G$ and the complete graph $T$ of order $2$ with $V(T)=\{a,c\}$ by adding the edges from the set $\{ax \, \colon \,x\in V(G)\}$. Every TD-set of $G^*$ has to contain the vertex $a$ in order to totally dominate the vertex $c$. However, any minimal TD-set containing $a$ can only have two vertices, which yields $\Gamma_t(G^*)=2$. On the other hand, we observe that $\grz(G^*) \ge \grz(G)+1$, which can be arbitrarily large.

\section{Graphs with $Z(G)=\delta(G)$ and power domination}
\label{sec:powerdom}

A trivial lower bound on the zero forcing number in terms of the minimum degree came from the original paper on zero forcing due to the AIM-Group~\cite{AIM}, which showed that
\begin{equation}
\label{eq:ZvsDelta}
Z(G)\ge \delta(G)
\end{equation}
holds for all graphs $G$. Simple examples where the bound is attained are paths, since $Z(P_n)=1$. In this section, we will characterize the graphs $G$ that attain the bound, which is equivalent to satisfying $\grz(G)=n(G)-\delta(G)$, using a connection with the concept of power domination. Furthermore, we find a characterization of all graphs with power domination equal to~$1$.

As remarked earlier, the concept of power domination was introduced in~\cite{haynes-2002},  where it was motivated by the problem of monitoring an electrical power network. Next, we present its definition.

Power domination is a graph searching process which starts by placing phase measurement units on a set $S$ of vertices in the power network, which are then labeled as observed (for the purpose of relating the process with zero forcing we will call these vertices blue).  Now, the searching process consists of the following two steps, where the Propagation Step can be repeated:
\begin{enumerate}
\item[(1)]
Initialization Step (Domination Step): All vertices in $S$ as well as all neighbors of vertices in $S$ are observed (i.e., colored blue).
\item[(2)]
Propagation Step (Zero Forcing Step): Every vertex which is the only unobserved (i.e., non-blue) neighbor of some observed (i.e., blue) vertex becomes observed (i.e., blue).
\end{enumerate}
If eventually the whole network is observed (that is, all vertices become blue), $S$ is called a \emph{power dominating set}. The minimum cardinality of a power dominating set of a graph G is the \emph{power domination number}, and is denoted by $\gamma_P(G)$.

The following lemma provides an interesting relation between certain extremal families of graphs in power domination and in zero forcing.

\begin{lemma}\label{l:powerDomination}
If $G$ is a graph of order~$n$ with minimum degree $\delta$, then $Z(G)=\delta$ if and only if $\gamma_P(G)=1$ and there exists a power dominating set $\{x\}$ such that $\deg_G(x)=\delta$.
\end{lemma}
\begin{proof} 
First, assume that $Z(G)=\delta$. Therefore, there exists a zero forcing set $S$ with $\delta$ vertices, and let $x\in S$ be the vertex with which the color-change rule starts. Clearly, $x$ has $\deg_G(x)$ neighbors, and exactly one neighbor of $x$, say $y$, needs to be non-blue before the propagation process starts. Since $|S|=\delta$, this implies that $\deg_G(x)=\delta$ and $S\subset N_G[x]$. In addition, $S'=\{x\}$ is a power dominating set, since once the initialization step is over, $N_G[x]$ is dominated, and the process of propagation in power domination is the same as the zero forcing process.

Conversely, let $S'=\{x\}$ be a power dominating set of $G$, where $\deg_G(x)=\delta$. Clearly, $N_G[x]$ is a zero forcing set. In addition, after removing a neighbor $y$ of $x$ from $N_G[x]$, the set $S=N_G[x]\setminus \{y\}$ is also a zero forcing set of $G$. Indeed, the color-change rule can be applied on the blue vertex $x$ with only one non-blue neighbor, which is vertex $y$, after which the same propagation rules can be used (in either power domination and zero forcing). This implies that $Z(G)\le |N_G[x]|-1=\delta$, and we know that $Z(G)\ge \delta$ in any graph $G$. Thus $Z(G)=\delta$ as claimed. \QED
\end{proof}

\bigskip
Graphs $G$ with power domination number $\gamma_P(G)=1$ were studied in~\cite{SKP-21}, where many families of such graphs were presented, within the class of regular graphs of high degree (at least $n(G)-4$). Next, we give a complete characterization of graphs $G$ with $\gamma_P(G)=1$. The idea for the class of graphs that characterizes these graphs was inspired by Row's construction~\cite{Row-12}, where  graphs with zero forcing number equal to~$2$ were characterized. First, we introduce the following notations. For a path $P \colon a_1,a_2,\ldots , a_k$ and vertex $a_i$ of the path $P$ we denote the subpath of $P$ from $a_{i+1}$ to $a_k$ by $^{a_i}\!P:a_{i+1},a_{i+2},\ldots , a_k$.

A graph $G$ is a \emph{graph of  $k$ internally parallel paths} if there exists $x \in V(G)$, which is an end-vertex of each of $k$ internally pairwise disjoint induced paths $Q_1,\ldots , Q_k$, where $V(G)=\cup_{i=1}^{k}{V(Q_i)}$, and $V(Q_i) \cap V(Q_j)=\{x\}$ for any $i \ne j$, by possibly adding any number of edges between different paths, provided that the following property holds:
\begin{itemize}
    \item for every set of $\ell$ vertices, $x_{i_1} \in V(Q_{i_1}),\ldots , x_{i_\ell} \in V(Q_{i_\ell})$, where $\ell\le k$, each belonging to distinct paths and none of them being an end-vertex of the corresponding path, it holds that there exists a vertex $x' \in \{x_{i_1},\ldots , x_{i_\ell}\}$ with $\deg_Y(x')=1$, where $Y=V(^{x_{i_1}}\!Q_{i_1})\cup \cdots \cup \, V(^{x_{i_\ell}}\!Q_{i_\ell})$.
\end{itemize}

For example, the graph $G$ illustrated in Figure~\ref{fig:parallel} is a graph of three internally parallel paths.

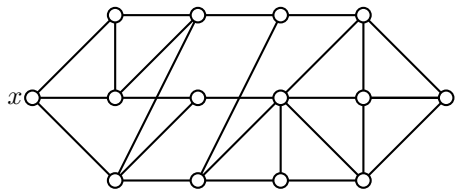
\begin{figure}[htb]
\begin{center}
\begin{tikzpicture}[scale=1.1,style=thick,x=1cm,y=1cm]
\def\vr{2.5pt} 

\draw (0,1)--(5,1);
\draw (0,1)--(1,0);
\draw (0,1)--(1,2);
\draw (1,0)--(4,0);
\draw (1,2)--(4,2);
\draw (4,2)--(4,1);
\draw (4,0)--(4,1);
\draw (1,1)--(1,2);
\draw (1,1)--(2,2);
\draw (1,0)--(2,1);
\draw (1,0)--(2,2);
\draw (2,0)--(3,1);
\draw (2,0)--(3,2);
\draw (3,0)--(3,1);
\draw (3,1)--(4,2);
\draw (3,1)--(4,0);
\draw (4,1)--(5,1);
\draw (4,2)--(5,1);
\draw (4,0)--(5,1);

\draw (1,0) [fill=white] circle (\vr);
\draw (1,1) [fill=white] circle (\vr);
\draw (1,2) [fill=white] circle (\vr);
\draw (2,0) [fill=white] circle (\vr);
\draw (2,1) [fill=white] circle (\vr);
\draw (2,2) [fill=white] circle (\vr);
\draw (0,1) [fill=white] circle (\vr);
\draw (3,0) [fill=white] circle (\vr);
\draw (3,1) [fill=white] circle (\vr);
\draw (3,2) [fill=white] circle (\vr);
\draw (4,0) [fill=white] circle (\vr);
\draw (4,1) [fill=white] circle (\vr);
\draw (4,2) [fill=white] circle (\vr);
\draw[anchor = east] (0,1) node {$x$};
\draw (5,1) [fill=white] circle (\vr);

\end{tikzpicture}
\caption{A graph $G$ of three internally parallel paths}
\label{fig:parallel}
\end{center}
\end{figure}

\begin{thm}\label{thm:PowerDomination1}
Let $G$ be a graph. Then $\gamma_P(G)=1$ if and only if $G$ is a graph of $k$ internally parallel paths for some $k$.
\end{thm}
\begin{proof}
Let $G$ be a graph with $\gamma_P(G)=1$, and let $\{x\}$ be a power dominating set of $G$. Let $k=\deg(x)$, and $N(x)=\{v_1^1,\ldots,v_1^k\}$. We have that $N[x]$ is a zero forcing set. So, in every propagation step, an uncolored vertex $w$ is turned into blue, where $w$ is the only uncolored vertex of a blue vertex. Note that if $N[x]=V(G)$, then $G$ is a graph with $k$ internally parallel graphs, where each one of these paths is $Q_i \colon x,v_1^i$. Hence we may assume that $N[x] \ne V(G)$, for otherwise the desired result hold.

Let $X=\{v_1^1,\ldots,v_1^k\}$. Let $w$ be the vertex colored in the first propagation step. Thus, there exist some $v_1^i$ such that $w$ is the only uncolored neighbor of $v_1^i$. We let $v_2^i=w$, and update $X$ to be $X=(X\setminus\{v_1^i\})\cup \{v_2^i\}$. We note that all the vertices in $X$ are blue and the blue vertices that are not in $X$ have all their neighbors colored blue at this point. This property will be maintained throughout the process.

We continue the process in the following way. In any propagation step we have a set $X=\{v_{i_1}^1,\ldots,v_{i_k}^k\}$, and let $w$ be the vertex that turns into blue in the present step. Then, there is a blue vertex for which $w$ is the only uncolored neighbor. Since, by assuming the mentioned property, any blue vertex that is not in $X$ has all its neighbors colored blue, then $w$ is the only uncolored vertex of some $v_{i_j}^{j}\in X$. We let $v_{i_j+1}^{j}=w$ and update $X$ to be $X=(X\setminus\{v_{i_j}^j\})\cup\{v_{i_j+1}^j\}$. We note that $X$ has $k$ blue vertices, and the blue vertices that are not in $X$ have all their neighbors colored blue, thus the desired property is maintained. We continue until all the vertices of $G$ are turned into blue.

At the end of the process, paths $Q_i \colon x,v_1^i,\ldots,v_{n_i}^i$, where $i\in[k]$, have been determined. To prove that $G$ is a graph of $k$ internally parallel paths, we need to verify the additional condition that these paths must satisfy.

Now, let $\{v_{i_1}^{j_1},v_{i_2}^{j_2},\ldots,v_{i_\ell}^{j_\ell}\}$, where $\ell \le k$, be a set  of vertices each belonging to pairwise distinct paths and $i_m<n_m$ for each $m\in[\ell]$. Let us assume, without loss of generality, that they were used to color an uncolored neighbor in the order $v_{i_1}^{j_1},v_{i_2}^{j_2},\ldots,v_{i_\ell}^{j_\ell}$. We remark that in the propagation step in which $v_{i_1}^{j_1}$ is chosen to color a neighbor, $v_{i_1+1}^{j_\ell}$ is the only uncolored neighbor of $v_{i_1}^{j_1}$. Moreover, all the vertices $v_{i_m}^{j_m}$ with $m>1$ are uncolored at this point. We note also that if $v_i^j$ is uncolored, then all the vertices in $^{v_{i}^j}Q_j$ are uncolored. As a consequence of these observations, all the vertices in $^{v_{i_m}^{j_m}}Q_{j_m}$ are uncolored for $m\ge 1$. So, we have that $v_{i_1+1}^{j_1}$ is the only neighbor of $v_{i_1}^{j_1}$ in
\[
Y=\bigcup_{m\in [\ell]} V(^{v_{i_m}^{j_m}}Q_{j_m}).
\]

Therefore, $G$ is a graph of $k$ internally parallel paths. Now, let $G$ be a graph of $k$ internally parallel paths with paths $Q_1 \colon x,v_1^1,\ldots,v_{n_1}^1$ through to $Q_k \colon x,v_1^k, \ldots,v_{n_k}^k$. We will prove that $\{x\}$ is a power dominating set. Equivalently, we need to prove that $\{x,v_1^1,\ldots,v_1^k\}$ is a zero forcing set. If $N[x]=V(G)$, we are done. Otherwise, let $X=\{v_1^1,\ldots,v_1^k\}\setminus\{v_{n_1}^1,\ldots,v_{n_k}^k\}$. Since $G$ is a graph of $k$ internally parallel paths, there is a vertex $v_{1}^i$ such that $\deg_Y(v_1^i)=1$, where
\[
Y=\bigcup_{j\in [k]}{V(^{v_1^j}Q_j)},
\]
and the neighbor is $v_2^1$. Then, we color $v_2^i$ blue, since it is the only uncolored neighbor of $v_1^i$. We continue this propagation process as follows. Let $X'$ be the set of the last colored vertices in each path $Q_i$, and $X=X'\setminus\{v_{n_1}^1,\ldots,v_{n_k}^k\}=\{v_{i_1}^{j_1},\ldots,v_{i_\ell}^{j_\ell}\}$. If $X\ne \emptyset$, since $G$ is a graph of $k$ internally parallel paths, there is a vertex $v_{i_m}^{j_m}$ such that $\deg_Y(v_{i_m}^{j_m})=1$, where
\[
Y=\bigcup_{t\in [\ell]}{V(^{v_{i_t}^{j_t}}Q_{j_t})},
\]
and the neighbor is $v_{i_m+1}^{j_m}$. Thus, since it is the only uncolored neighbor of $v_{i_m}^{j_m}$, we turn it into blue. We can continue until all the vertices in $G$ are blue, and therefore $\{x\}$ is a power dominating set and $\gamma_P(G)=1$.
    \QED
\end{proof}

\bigskip
Extremal graphs for the bound~\eqref{eq:ZvsDelta} are known for graphs with minimum degree $1$ or~$2$. The only graphs with $\delta(G)=1$ and $Z(G)=1$ are paths~\cite{ECK-17}. It was proved in~\cite{Row-12} that $Z(G)=2$ if and only if $G$ is a graph of two parallel paths, or, equivalently, $G$ is an outerplanar graph with the path cover number equal to $2$.
In our context, we are interested in those graphs with $Z(G)=2$ that have $\delta(G)=2$. One can derive that the only graphs of two such parallel paths (i.e., outerplanar graphs with path cover number equal $2$) with minimum degree $2$ are the $2$-connected outerplanar graphs. (The proof is straightforward, but we omit the details, since it follows from Corollary~\ref{cor:BigDelta}.)

\begin{cor}
If $G$ is a graph with $\delta(G)=2$, then $Z(G)=\delta(G)$ if and only if $G$ is a 2-connected outerplanar graph.
\end{cor}

For graphs $G$ with $\delta(G) \ge 3$, it is worth noting that the idea arising from graphs with two parallel paths is extended by the definition of graphs with $k$ internally parallel paths, where $k$ can be arbitrarily large.

From the proof of Theorem \ref{thm:PowerDomination1}, it follows that $G$ is a graph of $k$ internally parallel graphs if and only if $\gamma_P(G)=1$ and there exists a power dominating set $\{x\}$ with $\deg(x)=k$, and in consequence $Z(G)\le k$. Combining this with Lemma \ref{l:powerDomination}, we have a characterization of graphs with $Z(G)=\delta(G)$, or, equivalently $\grz(G)=n(G)-\delta(G)$.

\begin{cor}\label{cor:BigDelta}
    Let $G$ be a graph. Then $Z(G)=\delta(G)$ if and only if $G$ is a graph of $\delta(G)$ internally parallel paths.
\end{cor}


\section*{Acknowledgments}

The first and the third author were supported by the Slovenian Research and Innovation agency (grants P1-0297, J1-2452, J1-3002, and J1-4008). The second author was partially supported by grants PIP CONICET 1900, PICT-2020-03032 and PID 80020210300068UR. Research of the fourth author was supported in part by the South African National Research Foundation under grant number 132588 and the University of Johannesburg. The second and fourth authors thank the University of Maribor for the hospitality.


\end{document}